\documentclass[12pt]{article}
\usepackage{amsmath} 
\usepackage{amssymb} 
\usepackage{graphicx} 
\usepackage{enumerate}
\usepackage[utf8]{inputenc}
\usepackage{amsthm}

\newtheorem{thm}{Theorem}
\newtheorem{cor}[thm]{Corollary}

\newtheorem{rem}[thm]{Remark}

\newtheorem{ex}[thm]{Example}

\newcommand{\F}{\mathbb{F}}
\newcommand{\Fq}{\mathbb{F}_q}
\newcommand{\Fp}{\mathbb{F}_p}
\newcommand{\On}{{\bf{O}_n}}

\newcommand{\Pqn}{{\mathcal{P}}_{q,n}}
\newcommand{\Fqn}{\mathcal{F}_{q,n}}
\newcommand{\Cqd}{\mathcal{C}_{q,d}}

\begin{document}

\title{A Note on Value Sets of Polynomials over Finite Fields}

\date{\empty}
\author{Leyla I\c s\i k,  Alev Topuzo\u glu}

\maketitle

\noindent
Sabanc\i~ University, Orhanl\i, 34956 Tuzla, \.Istanbul, Turkey\\
E-mail: \{isikleyla, alev\}@sabanciuniv.edu\\

\begin{abstract}

Most results on the value sets $V_f$ of polynomials $f \in \mathbb{F}_q[x]$ relate the cardinality $|V_f|$ to the degree of $f$. 
In particular, the structure of the spectrum of the class 
of polynomials 
of a fixed degree $d$
is rather well known. 

We 
consider a class $\mathcal{F}_{q,n}$ of polynomials, which we obtain by modifying linear permutations at $n$ points. 
The study of the spectrum of $\mathcal{F}_{q,n}$ enables us to obtain a simple description of polynomials $F \in \mathcal{F}_{q,n}$ with prescribed $V_F$, especially 
those avoiding a given set, like cosets of subgroups of the multiplicative group $\mathbb{F}_q^*$. 
The 
value set count for such $F$ can also be determined. This yields 
polynomials with evenly distributed values, 
which have small maximum count. 

\end{abstract}





\section{Introduction}

Let $\Fq$ be the finite field with $q=p^r$ elements, where $p$ is a prime, $r \ge 1$.  
The {\it value set} of a polynomial $f \in \Fq[x]$ is the set $V_{f}= \{f(c): c\in \Fq\}$. 
Given a class of polynomials $ \mathcal{C}$, the set $v( \mathcal{C}) = \{|V_{f}| : f\in  \mathcal{C}\}$ is called the {\it spectrum} of $ \mathcal{C}$. Throughout we assume that polynomials are reduced $\mod x^q-x$, i.e. have degree $\leq q-1.$ 

Most previous results on value sets are concerned with the class $\Cqd$ of polynomials over $\Fq$ of degree $d$. When $d\leq 4$, the complete spectrum $v(\Cqd)$ is known, see for instance Section 8.3 of \cite{MuPa2013}. 
When $f \in \Cqd$ is not a permutation, i.e. $|V_{f}|< q$, Wan showed in \cite{Wan1992} that
\begin{equation}\label{eqn:wan}
|V_f|\leq q-\left \lceil \frac{q-1}{d}\right \rceil.
\end{equation}

For many interesting results on $v( \Cqd)$, we refer to \cite[8.3.3]{MuPa2013} and \cite[8.2]{shparlinski} and the references therein. For instance, it is shown in \cite {Gomez88} that the gap described in (\ref{eqn:wan}) between permutations and non-permutations is not the only gap in $v(\Cqd)$.  

In a recent paper \cite{MuWaWa2014}, Mullen et al. obtained an upper bound, similar to that in (\ref{eqn:wan}), for non-permutation polynomials $f$ of a fixed {\it index} $\ell$; 
$$|V_f|\leq q-\frac{q-1}{\ell}. $$
We recall that 
any non-constant polynomial $f\in \Fq[x]$ of degree $\leq q-1$ can be written uniquely as $f(x)=a\big(x^{r}g(x^{(q-1)/\ell)})\big)+b$, with a monic polynomial $g(x)$, where $\ell$  is the index of $f$, see \cite{AkGhWa2009}. 

The well-known Lagrange Interpolation Formula is an explicit formula for polynomials with any given value set. On the other hand, giving a simple description of polynomials with specific value sets, and/or determining their properties have been of interest. Bir\'{o}, for instance, shows in \cite{Biro2000} that polynomials over the prime field $\Fp$ taking only two non-zero values have degree $\geq 3(p-1)/4,$ except for some special polynomials. The so-called {\it minimal value set polynomials} $f \in \Cqd$ are those, satisfying $|V_f|=\lceil{q/d}\rceil$; the 
minimum value in $v(\Cqd)$.  
They are obtained by a variety of methods, see for example \cite{BoCo2013, CaLeMi1961, ChGoMu1988, GoMa1988}.  
Chou et al. study several classes of polynomials, the value sets of which lie in a subfield, see \cite{Chou}. Cusick, in \cite {Cusick}, determines five kinds of polynomials over
$\F_{2^n}$ with evenly distributed values, by using results from the theory of crosscorrelation of binary $m$-sequences. The 
distribution of the values of a polynomial is described in terms of the
{\it value set count} and {\it maximum count} (the definitions are given in Section 2).
It is shown for instance that if $n \equiv 0$ mod $4$, and
$f(x)= (x^{2^{n/2}}+x)^2/(x^2+x)$  is in $\F_{2^{n}}[x]$, then $0$ appears  $2^{n/2}$ times in $V_f$, regarded as a multiset, while all the other
elements appear twice, see \cite {Cusick}.
For further results on polynomials with particular value sets, we refer the reader to \cite[8.3.3]{MuPa2013}. 

In this note we consider a class $\mathcal{F}_{q,n}$ of polynomials over $\Fq$, $q \geq 5$, which we obtain by modifying linear permutations at $n$ points, where $q$ is an odd prime power.
As one might expect, the spectrum of $\mathcal{F}_{q,n}$ is of a very different nature, when compared with that of $\Cqd$. We give a simple description of polynomials $F \in \mathcal{F}_{q,n}$ with small and large value sets, for example with $|V_F|=2, 3, 4, q-3,q-2.$ We determine
$V_F$ explicitly for such $F$. We also show how to obtain polynomials, whose values avoid prescribed sets, in particular any coset of any subgroup of the multiplicative group $\Fq^*$. We give value set counts 
for polynomials $F$ and 
hence find polynomials with evenly distributed values, extending some results of Cusick mentioned above, 
to odd characteristic: See Corollary \ref{complete} \textit{(ii)}, \textit{(iii)} and Remark \ref{maximumcount}. We note that the polynomials in $\mathcal{F}_{q,n}$ may also be permutations, giving rise to
complete mappings, see for instance Theorem \ref{n=4} and Remark \ref{comp}.

We define $\Fqn$, $n \geq 2$ as follows. Start with $g(x)=ax+b \in \Fq[x], ~a, b \neq 0.$
Let $\On$ denote a set of $n$ distinct elements $x_1, x_2, \ldots x_n \in \Fq$ with $x_1=0$ and $x_n=-b/a$.
Consider the set $\Pqn$ of permutations $f$ of $\Fq$ defined by,

\begin{eqnarray}
\label{eq:Pqn}
f(x)=
\left \{
\begin{array}{lll}
g(x) & \mbox{if} \quad x \notin \On,\\
g(x_{i-1}) & \mbox{if} \quad x=x_i \in \On, ~2 \leq i \leq n,\\
g(x_n)=0 & \mbox{if}  \quad x=x_1 \in \On.

\end{array}
\right.
\end{eqnarray}

 Adding the identity permutation to permutations in $\Pqn$, we obtain
 polynomials with a variety of value sets. 
 We put 
$\Fqn= \{F(x)=f(x)+x: f \in \Pqn \}.$ 

This particular choice of the polynomials $f$ (and $F \in \Fqn$) enables us to use the tools that are obtained 
in \cite {aaaw2009}. Indeed for any $f \in \Pqn$, defined as in (\ref{eq:Pqn}), one can uniquely determine a polynomial
$P_n(x)=P_{n}(c_{0},\ldots,c_{n};x)$, which satisfies the following properties. One has  
\begin {equation}
\label{eq:f}
f(\delta)=P_{n}(c_{0},\ldots,c_{n};\delta)=(\ldots((c_0\delta)^{q-2}+c_1)^{q-2} \ldots +c_n)^{q-2}
\end{equation}
for all $\delta \in \Fq$, where $c_{0},\ldots,c_{n} \in \Fq^*$. The recursively defined sequences
\begin{equation}
\label{alpha-beta}
\alpha_{k}=c_{k-1}\alpha_{k-1}+\alpha_{k-2}, \hspace{.2in} 
\beta_{k}=c_{k-1}\beta_{k-1}+\beta_{k-2},~~k \geq 2,
\end{equation}
with $\alpha_{0}=0, \alpha_{1}=c_{0}, \beta_{0}=1,
\beta_{1}=0,$ point to the relation between $f$ and $P_n(x)$ as $a=\alpha_n/\beta_{n+1},~~ b= \beta_n/\beta_{n+1}$. Moreover they satisfy $ \alpha _{n+1}=0$, 
 and $\alpha_k \neq 0$ when $1 \leq k \leq n$. 

It is easy to see that using the representation (\ref{eq:f}) for $f$, the set 
$\On$ above can also be expressed as 
$\On= \{x_{i}:
x_{i}=\frac{-\beta_{i}}{\alpha_{i}}, \; i=1, \ldots,n \}$.
A procedure that yields 
$P_{n}(c_{0},\ldots,c_{n};x)
$ from a given polynomial
$g(x)$ and the set $\On$ is essentially given in \cite{CeMeTo2008}, and will be summarized in Section 4. See also Example \ref{Ex2}, below. 

Conversely, consider a permutation polynomial 
$P_n(x)=P_{n}(c_{0},\ldots,c_{n};x)=(\ldots((c_0x)^{q-2}+c_1)^{q-2} \ldots +c_n)^{q-2}~\in \Fq[x]$,
and the elements $\alpha_k, \beta_k, 1 \leq k \leq n,$ which are defined recursively by 
(\ref{alpha-beta}). Put $\mathbf{\overline{O}_{n}}=\{x_{i}:
x_{i}=\frac{-\beta_{i}}{\alpha_{i}}, \; i=1, \ldots,n \}$. If $|\mathbf{\overline{O}_{n}}|=n$, 
$\alpha_k \neq 0$ for $1 \leq k \leq n$ and $\alpha _{n+1}=0$, then there are uniquely determined
polynomials $g(x)$ and $f(x)$, where $g(x)=(\alpha_n/\beta_{n+1})x+\beta_n/\beta_{n+1}$, $\On=\mathbf{\overline{O}_{n}}$, and
$f(x)$ as defined in (\ref{eq:Pqn}) that satisfies $f(\delta)=P_n(\delta)$ for every $\delta \in \Fq$. We refer the reader to \cite{aaaw2009} for details. Note that
in the terminology of \cite{aaaw2009}, the set 
$$ 
\mathbf{\overline{O}_{n+1}}=\mathbf{\overline{O}_{n}}\cup \{x_{n+1}\}= 
\{x_{i}:
x_{i}=\frac{-\beta_{i}}{\alpha_{i}}, \; i=1, \ldots,n \} \cup \{ \infty \} \subset
\mathbb{F}_{q} \cup \{ \infty \}
$$
is called the set of {\it {poles}}, where $x_{n+1}$ is the pole at infinity.

A representation similar to that in (\ref{eq:f}) is possible for any permutation $P$ of $\Fq$, and it leads to the concept 
of the \textit{Carlitz rank } of $P$. This notion was introduced in \cite{aaaw2009}, and is particularly
useful when $n$ is small with respect to $q$, since it enables expressing permutations as fractional linear transformations, except at at most $n+1$
elements of $\Fq$, see \cite{ GoOsTo2014, PaTo2014, Topuzoglu}.

\section{The Spectrum  $\mathnormal{v}(\mathcal{F}_{q,n})$}

\begin{thm}
The spectrum $\mathnormal{v}(\mathcal{F}_{q,n})$ satisfies
$$\mathnormal{v}(\mathcal{F}_{q,n}) \subset \left \{2,3,\ldots,n+1,q-n,q-n+1,\ldots,q-2,q\right \}.$$
\end{thm}

\begin{proof} 
Let $F (x) \in \Fqn$ be arbitrary. We recall that $a, b \neq0$,
\begin{equation}
\label{F}
F(\delta)= (a+1)\delta+b ~~\textrm{for}~ \delta \in \Fq \setminus \On, ~\textrm{and}~ F(x_i)=ax_{i-1}+x_i+b,~ i=2 \ldots n.
\end{equation}
We also note that 
$|\On|=n$ and $F(x_1)=F(0)=g(x_n)+x_1=g(-b/a)=0$, see (\ref{eq:Pqn}). 

Assuming $a=-1$, 
we get $F(\delta)=b=x_n$ for all $\delta \in \Fq \setminus \On$.
The cardinality of the image $F(\On)$ may vary between $1$ and $n$. 
The element $0$ is in $V_F$ since $F(x_1)=0$, while $x_n$ is also in $V_F$ and $x_n  \neq 0.$ 
Hence $2 \leq |V_F| \leq n+1.$ 

If $a \neq -1$, then $|F(\Fq \setminus \On)|=q-n$ and $|V_F|=q-n$ if $F(\On) \subset F(\Fq \setminus \On)$. Therefore depending on the intersection of the two sets $F(\On)$ and
$F(\Fq \setminus \On)$, the cardinality $|V_F|$ may range between $q-n$ and $q$, except that $|V_F| \neq q-1$ for any
$F \in \Fqn$.

Suppose that $|V_F|=q-1$, so that
$q-2$ elements appear once in $V_F$, regarded as a multiset, one element $\alpha$ appears twice and
one element $\beta$ does not appear. But in this case 
$\sum_{c \in \Fq} F(c)= \sum_{c\in \Fq} f(c)+\sum_{c\in \Fq} c=0,$
while
$
\sum_{c \in \Fq} F(c)=\alpha+\sum_{c\in \Fq\setminus\{\beta\}} c =\alpha-\beta.
$
 \end{proof}

\begin{ex}
We remark that all the values listed in Theorem 1 may actually be attained. This example with $q=13$, $n=6$ shows that
there are $F \in \mathcal{F}_{13,6}$, for which $ |V_F| $ takes any value between $3$ and $q-2=11$ and also $q=13.$ 
In what follows $F_i(x) \in \mathcal{F}_{13,6}$ denotes a polynomial with $|V_{F_i}|=i$
for $1 \leq i \leq 11$, and $i=13.$ We first fix a linear polynomial $g_i(x)$ and the set ${\bf{O}_6}^{(i)}$ to obtain $f_i(x)$ as in 
(\ref{eq:Pqn})
so that
$F_i(x)=f_i(x)+x \in \mathcal{F}_{13,6}$ satisfies $|V_{F_i}|=i$. We use the notation 
$P_6^{(i)}(c_o, \ldots, c_6;x)$ to specify the polynomial $P_6(x)$, satisfying $f_i(\delta)=P_6(\delta)=P_6^{(i)}(\delta)$
for any $ \delta \in \Fq.$
For instance the polynomial
$P_6^{(13)}(7,2,3,6,10,5,2;x)$, corresponding to $g_{13}(x)=5x+3$ and ${\bf{O}_6}^{(13)}=\{0,12,1,8,6,2 \} $ yields $F_{13}(x)$, such that
$|F_{13}({\bf{O}_6}^{(13)})|=6$, $|F_{13}(\Fq\setminus {\bf{O}_6}^{(13)})|=7$ and $F_{13}({\bf{O}_6}^{(13)})\cap F_{13}(\Fq\setminus {\bf{O}_6}^{(13)})=\emptyset$. 

The polynomials $g_{11}(x)= 7x+12,$ $ g_{10}(x)=9x+3,$ $ g_9(x)= x+4,$ $ g_8(x)=10x+12,$ with
${\bf{O}_6}^{(11)}=\{0,4,8,11,7,2 \},$ $ {\bf{O}_6}^{(10)}=\{0,6,11,1,7,4 \}$, ${\bf{O}_6}^{(9)}=\{0,11,12,4,3,9 \}$, $ {\bf{O}_6}^{(8)}=\{0,6,3,12,7,4\}$ produce examples of $F_i(x)$ with $|V_{F_i}|=i,$ $8\leq i\leq11$. The corresponding permutations are
$P_6^{(11)}(5,11,1,4,7,9,$\\$3;x)$, $P_6^{(10)}(1,2,8,3,7,6,1;x)$, $P_6^{(9)}(1,7,2,1,11,9,2;x)$, $P_6^{(8)}(1,2,7,8,5,4,10;x)$.

The other end of the spectrum, i.e. the values $3, 4, 5, 6, 7$ can be obtained by choosing 
$g_3(x)=12x+5$, $g_4(x)=12x+10$, $g_5(x)=12x+6$, $g_6(x)=12x+10$, $ g_7(x)=12x+8$ and ${\bf{O}_6}^{(3)}=\{0,8,3,2,10,5\}$, 
${\bf{O}_6}^{(4)}=\{0,1,4,7,9,10 \}$, ${\bf{O}_6}^{(5)}=\{0,3,10,2,9,6 \}$, ${\bf{O}_6}^{(6)}=\{0,8,2,11,3,10 \}$, ${\bf{O}_6}^{(7)}=\{0,4,7,9,10,8 \}$.  
The corresponding permutations
are
$P_6^{(3)}(1,8,3,9,4,10,5;x)$, $ P_6^{(4)}(4,3,1,5,12,$\\$5,1;x)$, $P_6^{(5)}(10,3,2,9,1,11,4;x)$, $P_6^{(6)}(9,11,2,7,8,4,2;x)$, $ P_6^{(7)}(1,3,5,4,2,11,\\$ $6;x)$. 

The value $|V_F|=2$ is attained when char$(\Fq)=p=n$, and the linear polynomial $g(x)$ and the set 
${\bf{O}_n}$ are chosen as in Theorem 9 (iv) below.   For instance, $P_5(-1,-1/b,2b,-2/b,2b,-1/b;x)$ over $F_{25}$ yields a polynomial $F$ with $|V_F|=2$, for any $b\in \mathbb{F}_{25}^*$.
\end{ex}

In order to analyse the spectrum $\mathnormal{v}(\mathcal{F}_{q,n})$ in more detail we follow \cite{Cusick}
and define the value set count and the maximum count.
For $f \in \Fq[x]$, the value set count is defined in terms of the pre-images of elements in $\Fq$.
It is the vector $(v_0, v_1, \ldots, v_M)$, where 
$v_i=|\{\alpha \in \Fq: |f^{-1}(\alpha)|=i \}|,$ and $M= \max_{\alpha \in \Fq}\{|f^{-1}(\alpha)|\}$
is the maximum count for $V_f.$ When $v_k=0$ for $1 < i \leq k \leq j < M,$ and $v_{i-1} \neq 0,~v_{j+1} \neq 0$, we write
$(v_0, v_1, \ldots, v_{i-1},v_{j+1},\ldots, v_M).$ When we wish to specify the elements in $V_f$ with a given \textit{multiplicity},
we use the notation $m(\beta)=|f^{-1}(\beta)|$ for $\beta \in V_f.$

\section{The cases $n=2,~3$}
When  $n=2, 3,$ it is easy to describe the polynomials $F \in \Fqn$ and and their value sets $V_F$ explicitely. 
\begin{thm}
The spectrum of the family $\mathcal{F}_{q,2}$ is 
$$\mathnormal{v}(\mathcal{F}_{q,2})= \{3, q-2\}.$$
\end{thm}
\begin{proof} Let $F \in \mathcal{F}_{q,2}$ be arbitrary with $F(x)=f(x)+x, ~f \in \mathcal{P}_{q,2}.$ 
It can be seen easily that $F$ can be represented as
\begin{equation}
\label{F2}
F(x)=\left(\Big((c_{0}x)^{q-2}+c_{1}\Big)^{q-2}-\frac{1}{c_{1}}\right)^{q-2}+x,
\end{equation}
where $a=-c_0c_1^2, ~b=-c_1,$
$x_2=-1/c_0c_1$ and $F(x_{2})=-(c_{0}c_{1}^{2}+1)/(c_{0}c_{1}).$
When $a=-1,$ 
one has 
$F(\delta)=-c_{1}$ for any $\delta \in \Fq^{*}\setminus\{x_{2}\}$, and $F(x_{2})=-2c_{1}$. 
Therefore
$V_{F}=\{0,-c_{1},-2c_{1}\}$. 

The cases $a=1$ and $a \neq \pm1$, or with the above notation $c_{0}c_{1}^{2}=-1$ and $c_{0}c_{1}^{2}\neq \pm1$ yield
polynomials $F$ in (\ref{F2}) with $|V_F|=q-2.$ By straightforward calculations one can see that
if $c_{0}c_{1}^{2}\neq \pm1$, the polynomial $F(x)$ in (\ref{F2}) has the value set
$V_F= \Fq \setminus \{ -c_1, -1/c_0c_1)\}$, and if $c_{0}c_{1}^{2}=-1$, then
$V_ F=\Fq\setminus\{c,-c\}$ where $c=c_1$. 

\end{proof}

\begin{cor} 
\label{rem:F2}
The following polynomials have value sets of cardinalities $3$ and $q-2.$
\begin{itemize}
\item[(i)]
Any polynomial $F \in \mathcal{F}_{q,2}$  of the form
$$F(x)=\left ( \left (\left (\frac{1}{c^{2}}x\right )^{q-2}+c\right )^{q-2}-\frac{1}{c}\right )^{q-2}+x,$$
where $c\in \Fq^*$, has the value set $V_F=\{0,-c,-2c\}$. 
The value set count for $F$ is $(v_0, v_1, v_{q-2})$,
where $v_0=q-3, v_1=2, v_{q-2}=1$.
\item[(ii)]
Any
polynomial $F \in \mathcal{F}_{q,2}$  of the form
$$F(x)=\left ( \left (\left (\frac{-1}{c^{2}}x\right )^{q-2}+c\right )^{q-2}-\frac{1}{c}\right )^{q-2}+x,$$
has the value set $V_F=\Fq\setminus\{c,-c\}$, where $c \in \Fq^*$ is arbitrary. 
The value set count in this case is $(v_0,v_1,v_3)$ where $v_0=2$, $v_1=q-3$, 
and $v_3=1$. 
\end{itemize}
\end{cor}

\begin{proof}
Immediate from the above proof by putting $c=c_1.$
\end{proof}

\begin{thm} 
\label{n=4}
The spectrum of the family $\mathcal{F}_{q,3}$ is
\begin{displaymath}
v({\mathcal{F}}_{q,3})=
\left \{
\begin{array}{ll}
\{2,4,q-3, q-2, q\} & \mbox{if} \quad q\equiv 0 \mod 3,\\
\{3,4,q-3, q-2,q\} & \mbox{if} \quad q\equiv 1 \mod 3,\\
\{3,4,q-3,q-2\} & \mbox{if}  \quad q\equiv 2 \mod 3.

\end{array}
\right.
\end{displaymath}
\end{thm}
\begin{proof}
When $F \in \mathcal{F}_{q,3}$, the corresponding polynomial
$f(x)=P_3(c_0,c_1,c_2,c_3;x)$ in (\ref{eq:f}) satisfies
$a=c_0(c_1c_2+1)^{2}, ~b= c_2(c_1c_2+1), ~c_3= -c_1/(c_1c_2+1).$

We first focus on small value sets, i.e., $|V_F|=2,3,4$. As we have seen in Theorem 1, these values occur when $a=-1$ or $c_0=-1/(c_{1}c_{2}+1)^{2}$.
In this case $x_2=(c_1c_2+1)^{2}/c_1,~x_3=b=c_2(c_1c_2+1)$ and hence we get 
$F(x_2)=(c_1c_2+1)(2c_1c_2+1)/c_1,$ 
$F(x_3)=(c_1^2c_2^{2}-1)/c_1 $. Since $F(\Fq\setminus \mathbf{O_3})=\{x_3\}$, we have   $F(\mathbf{O_3})\cap F(\Fq\setminus \mathbf{O_3})=\emptyset.$ 

One can easily see that 
$|F(\mathbf{O_3})|\leq 3$ exactly when 
$c_1c_2 \in \{-2,-1/2,1\}$. For instance
when $c_1c_2=1$,  one gets $F(x_1)=F(x_3)=0$ and $F(x_2)=6 c_2$ while $F(\delta)=2c_2$ for 
$\delta \in \Fq \setminus \mathbf{O_3}$.  Hence  for char$(\Fq) \neq 3$ we obtain $V_F=\{0,2c_2,6c_2\}$ with the value set count $(v_0,v_1,v_2,v_{q-3})$, where $v_0=q-3$, $v_1=v_2=v_{q-3}=1$. 

For $c_1c_2=-1/2$ one gets $F(x_1)=F(x_2)=0$, $F(x_3)=3c_2/2$, which yields $V_F=\{0,3c_2/2,c_2/2\}$, since
$F(\delta)=c_2/2$ for every $\delta \in \Fq\setminus\mathbf{O_3}$. Note that $|F^{-1}(0)|=2$, $|F^{-1}(3c_2/2)|=1$ and $|F^{-1}(c_2/2)|=q-3$ when char$(\Fq)\neq 3$ and hence the value set count is $(v_0,v_1,v_2,v_{q-3})$, where $v_0=q-3$, $v_1=v_2=v_{q-3}=1$.
The case $c_1c_2=-2$ can be dealt with similarly to yield $V_F=\{0,-3c_2/2,-c_2\}$, so that $|V_F|=2$ or $3$ depending on char$(\Fq)$.   

When $d=c_1c_2 \notin \{-2,-1/2,0,1\}$, one obtains polynomials $F$ with the value set 
$V_F=\{0,(d+1)(2d+1)/c,(d^{2}-1)/c,d(d+1)/c\}$, with $c=c_1$.

For the large value sets of sizes $q-3,q-2$ and $q$ that are obtained when $c_0 \neq -1/(c_1c_2+1)^2$, we consider the cases 
\begin{itemize}
\item[(i)] $c_0=-1/(c_1c_2+1)$,  implying $F(x_1)\cap F(\Fq\setminus \mathbf{O_3})=\emptyset$,
\item[(ii)] $c_0=1/c_1c_2(c_1c_2+1)^2$,  implying $F(x_2)\cap F(\Fq\setminus \mathbf{O_3})=\emptyset$,
\item[(iii)] $c_0=-(c_1c_2)/(c_1c_2+1)^3$,  implying $F(x_3)\cap F(\Fq\setminus \mathbf{O_3})=\emptyset$.
\end{itemize}

When $q \equiv1\mod 3$ or char$(\Fq)=3$ the equation $c_1^2c_2^2+c_1c_2=-1$ has a solution and hence the equations (i),(ii),(iii) above are simultaneously satisfied. By straightforward calculations one can see in this case that $|F(\mathbf{O_3})|=3$, and with $F(\mathbf{O_3})\cap F(\Fq\setminus \mathbf{O_3})=\emptyset$, one obtains $|V_F|=q$. 

When $q\equiv 0 \mod 3$ and $c_0=1/(c_{1}c_{2}+1)^3$ we have $F(x_2)=F(x_3)=F(\delta_1)$ for a unique $\delta_1 \in \Fq\setminus\mathbf{O_3}$. Moreover, if $(c_{1}c_{2}+1)^2 \neq -1$ then there exists a unique $\delta_2\in \Fq \setminus \mathbf{O_3}$ with $F(x_1)=F(\delta_2)$. Therefore we get $|V_F|=q-3$.  In this case $V_F=\Fq\setminus\{cd(d-1),cd(-d^2-1),cd^2(-d+1)\}$ where $d=c_1c_2+1$ and $c=1/c_1$ is arbitrary in $\Fq^*$. Otherwise if $(c_{1}c_{2}+1)^2 = -1$, that is $q\equiv 9 \mod 12$, then we have $|V_F|=q-2$.  In this case we have $V_F=\Fq\setminus\{c(-1-d),c(d-1)\}$ where $d=c_1c_2+1$ and $c=1/c_1$ is arbitrary in $\Fq^*$. 

When $q\equiv 1 \mod 3$, $c_{0}=1/(c_{1}c_{2}+1)^2$ and $c_{1}c_{2} \notin\{-2,-1/2,1\}$ then there are uniquely determined  $\delta_1,\delta_2,\delta_3 \in \Fq \setminus \mathbf{O_3}$ with $F(x_1)=F(\delta_1)$, $F(x_2)=F(\delta_2)$ and $F(x_3)=F(\delta_3)$. Hence we get $|V_F|=q-3$. In this case $V_F=\Fq\setminus\{cd(d-1),cd(-d-1),cd(-d+1)\}$ where $d=c_1c_2+1$ and $c=1/c_1$ is in $\Fq^*$. On the other hand,  if $c_1c_2=1$ then there exist uniquely determined elements $\delta_1, \delta_2 \in \Fq \setminus \mathbf{O_3}$ with $F(x_1)=F(\delta_1)$, $F(x_3)=F(\delta_2)$ and all other values are attained once. In this case we get $V_F=\Fq\setminus\{2c,-6c\}$ where $c=1/c_1$ is in $\Fq^*$. Hence $|V_F|=q-2$.

When  $q\equiv 2\mod3$ we also get $|V_F|=q-2$ or $q-3$. There are various possibilities for the values $q-2$ or $q-3$ 
to be attained. For instance if $q \equiv 5 \mod 12$, 
$c_0=-1/(c_1c_2+1)$ and $(c_1c_2+1)^2=-1$, 
then there exists a unique $\delta \in \Fq \setminus \mathbf{O_3}$ with $F(\delta)=F(x_2)=F(x_3).$ In this case $V_F=\Fq\setminus\{-c-cd,-c+cd\}$ where $d=(c_1c_2+1)$ and $c=1/c_1$ is 
in $\Fq^*$. 

If, on the other hand, $q \equiv 11 \mod 12$, $c_0=1/c_1c_2(c_1c_2+1)$, and $c_1c_2 \notin\{ 0, -1, -1/2\}$,
then there exist uniquely determined $\delta_1,\delta_2 \in \Fq \setminus \mathbf{O_3}$ such that $F(\delta_1)=F(x_1)=F(x_2)=0$ and $F(\delta_2)=F(x_3)$. In this case $V_F=\Fq\setminus\{cd(d+1),-c(d+1)^2,-d^2c\}$ where $d=c_1c_2$ and $c=1/c_1$. is 
in $\Fq^*$. 

\end{proof}

\begin{cor}
\label{complete}
Depending on $q$, the following polynomials have value sets of cardinalities $q, q-2$, and $q-3$.

\begin{itemize}
\item[(i)] Let $q \equiv1\mod 3$ 
. 
Any polynomial $F(x)$ of the form 
$$F(x)= \left (\left (\left (\left (\dfrac{-x}{d+1}\right )^{q-2}+c\right )^{q-2}+\dfrac{d}{c}\right )^{q-2}-\dfrac{c}{d+1}\right )^{q-2}+x $$
is a permutation, if $d$ satisfies $d^2+d+1=0$, in other words $d$ 
is a primitive third root of unity, and $c \in \Fq^*$ is arbitrary. 
Note that $F(\delta)=f(\delta)+\delta$ for $\delta \in \Fq$, where $f$ is obtained by altering
$g(x)=-(d+1)x+(d^2+d)/c$ at three elements; $x_1=0, x_2=(d+1)/c, x_3=d/c$, as in (\ref{eq:Pqn}).
\item[(ii)] Let $q \equiv5\mod 12$. Any polynomial $F(x)$ of the form
$$F(x)= \left (\left (\left (\left (\dfrac{-x}{d+1}\right )^{q-2}+c\right )^{q-2}+\dfrac{d}{c}\right )^{q-2}-\dfrac{c}{d+1}\right )^{q-2}+x $$
with $(d+1)^2=-1, ~c \in \Fq^*$ arbitrary, 
has the value set $V_F=\Fq\setminus\{d/c, (-d-2)/c \}$. Here $m(-1/c)=3$, and $m(\alpha)=1$, for all $\alpha \in V_F, \alpha \neq -1/c.$
Hence the maximum value count is $3$.
\item[(iii)] Let $q \equiv11 \mod 12$. Any polynomial $F(x)$ of the form
$$F(x)= \left (\left (\left (\left (\dfrac{x}{d(d+1)}\right )^{q-2}+c\right )^{q-2}+\dfrac{d}{c}\right )^{q-2}-\dfrac{c}{d+1}\right )^{q-2}+x $$
has the value set $V_F=\Fq\setminus\{d(d+1)/c, -(d+1)^2/c, -d^2/c \}$, where
 $c \in \Fq^*$ is arbitrary and $d \notin \{-1, -1/2, 0 \}.$ Here 
$m((-d^2-d-1)c)=2, m(0)=3$ and $m(\alpha)=1$ for all the other $q-5$ elements $\alpha \in V_F.$
Hence the maximum value count is $3$.  
The value set count is $(v_0,v_1,v_2,v_3)$ where $v_0=3$, $v_1=q-5$, $v_2=1$, $v_3=1$.

\end{itemize}
\end{cor}

\begin{proof}
By putting $c_1c_2=d$ and $c=c_1$ the results follow immediateley from the proof of the above theorem. 
\end{proof}

\begin{rem}
\label{comp} 
The polynomial $F(x)$ 
 in \textit{(i)} in the above corollary, and $F(x)-x=f(x)$
 are both permutations, i.e., $f(x)$ is a \textit{complete} mapping.
 
 \end{rem}
 
 \begin{cor} Any polynomial $F(x)$ of the form
$$F(x)= \left (\left (\left (\left (\dfrac{-x}{(d+1)^2}\right )^{q-2}+c\right )^{q-2}+\dfrac{d}{c}\right )^{q-2}-\dfrac{c}{d+1}\right )^{q-2}+x $$
has the value set $V_F=\{0, (d^2-1)/c, (d+1)d/c, (d+1)(2d+1)/c \}$, where $c \in \Fq^*$ is arbitrary and $d \notin \{-2,-1/2, 0, 1 \}$.
The value set count is $(v_0,v_1,v_{q-3})$ with $v_0=q-4$, $v_1=3$, $v_{q-3}=1$.
\end{cor}

\section{Polynomials with prescribed $|V_f|$}

We first note that the polynomials $F \in \Fqn$ can be easily described as in (3), when $n$ is small enough and $\On$ is given. We summarize the procedure given in \cite{CeMeTo2008}, for the sake of completeness. Recall that $x_n=-b/a=-\beta_n/\alpha_n$. Given $g(x)=ax+b$, and $\On$, we can therefore use the equations (\ref{alpha-beta}) to calculate $c_0, c_1, \ldots, c_n$, and hence $F(x)=P_n(c_0, \ldots, c_n:x)+x.$ Put $\alpha_n=\epsilon a,~\beta_n=\epsilon b$, $\alpha_{n+1}=0$
and $\beta_{n+1}= \epsilon$, for a variable $\epsilon$. From (\ref{alpha-beta}) one gets 
$$c_{i}=\frac{\beta_{i+1}+x_{i-1}\alpha_{i+1}}{\beta_{i}+x_{i-1}\alpha_{i}},~2\leq i \leq n .$$
Therefore one can 
recursively calculate the exact
values for $c_{n},c_{n-1},\ldots,c_{3}$, and values for 
$\alpha_{n-1},\beta_{n-1},\ldots,\alpha_{2},\beta_{2}$ as multiples
of $\epsilon$. In the final step,  
using 
$\beta_{0}=1$, $\beta_1=0$,
and hence
$\beta_{2} =
1$, one obtains the value for $\epsilon$.
Then one can find $c_{2}=\beta_{3}$, $\alpha_1=\alpha_3-c_2\alpha_2$, $c_1=\alpha_2/\alpha_1$  and
$c_{0} =\alpha_{1}$. 

The following theorem gives examples of polynomials $F$ with
small and large value sets, depending on $q,~n$.

\begin{thm} 
\label{thm:q-n}
The following choices of $g(x)$ and $\On$ yield polynomials $F \in \Fqn$ with
$|V_F|= n+1,~|V_F|=q-n$, $|V_F| \geq q-n$, where the maximum count for $V_F$ is two,
 and $|V_F|=2,~3$.
\begin{itemize}
\item [(i)] Let $q=p^r, ~p>n(n+1)$. Then the polynomial $F(x)$ as in (\ref{F}),
with $a=-1,~ b=n(n-1)/2$ and 
$\On= \{x_i: x_1=0,~x_i=x_{i-1}+i-1, ~2 \leq i \leq n \}$ has $|V_F|=n+1.$
The value set count for $F$ is $(v_0, v_1, v_{q-n})$, where $v_0= q-n-1, ~v_1=n, ~v_{q-n}=1.$
\item [(ii)] 
Let $q \equiv 1 \mod n$, and $a \in \Fq^*$, with \hbox{\rm ord}$(-a)=n$. Then for any $b \in \Fq^*$ and
$\On= \{x_i: x_i=b\sum_{j=0}^{n-i}(-1/a)^{j+1}, ~1 \leq i \leq n  \}$,
the polynomial $F(x)$ as in (\ref{F}) has $|V_F|=q-n.$
The value set count for $F$ is $(v_0, v_1, v_{n+1})$, where $v_0= n, ~v_1=q-n-1, ~v_{n+1}=1.$
\item [(iii)] 
Let $q \equiv 1 \mod 2n, ~a, b \in \Fq^*$ with \hbox{\rm ord}$(a)= 2n.$ Put $z_i=ba^{i-1}$, and $x_i= (z_i-b)/(a+1)$ for $1 \leq i \leq n.$
Then the polynomial $F(x)$ as in (\ref{F}) has $|V_F| \geq q-n$, and the maximum count for $V_F$ is at most $2$.
\item [(iv)]
Let $q=p^r, ~p>n-2.$
Then the polynomial $F(x)$ as in (\ref{F}),
with $a=-1,~ b \in \Fq^*$ and 
$\On= \{x_i: x_i=(1-i)b, ~1 \leq i \leq n-1,~ x_n=b\}$ has $V_F=\{0, b, nb\}$, with $m(0)=n-1, ~m(b)=q-n, ~m(nb)=1.$

\end{itemize}
\end{thm}

\begin{proof} 
\begin{itemize}
\item [\textit{(i)}]
Note that $x_n=b=n(n-1)/2.$ Since $a=-1$, we have, as in Theorem 1, $F(\delta)=b$ for every $\delta \in \Fq \setminus \On.$
The choice of $\On$ gives $|F(\On)|=n$ and $b \notin F(\On).$ Indeed $V_F= \{0\} \cup \{ b+i: 0 \leq i \leq n-1\}.$ The value set count
also follows trivially.

\item [\textit{(ii)}]
We show that by this choice of $\On$ we get $F(\On)=\{0\}$ and $F(\On) \cap F(\Fq \setminus \On)=\{0\}.$
Firstly note that $x_n=-b/a$,  and ord$(-1/a)=n$ implies $x_1=0$, $|\On|=n$, as required.
Recall that $F(x_i)=ax_{i-1}+x_i+b$ for $2 \leq i \leq n.$ Hence
$$F(x_i)=a((-b/a)\sum_{j=0}^{n-i+1}(-1/a)^j)+((-b/a)\sum_{j=0}^{n-i}(-1/a)^j)+b. $$ 
By straightforward calculations one gets $F(x_i)=0$ for every $1 \leq i \leq n.$
On the other hand there exists $\delta=-b/(a+1) \in \Fq^*$ with $F(\delta)=0.$
Note that $\delta \notin \On$, since $-b/(a+1)=(-b/a)\sum_{j=0}^{n-i}(-1/a)^j$ for some $1 \leq i \leq n$
yields a contradiction. 
Recalling that  
$F$ is linear on $\Fq \setminus \On$, we obtain $|V_F|=q-n.$ Note that
$m(0)=n+1$, and $m(\alpha)=1$ for all the other elements $\alpha \in V_F \setminus \{0\}.$

\item [\textit{(iii)}]
By this choice of $\On$, it is obvious that $F(x_i)\neq F(x_j)$ for $1\leq i \neq j \leq n.$
Therefore the only elements in $V_F$
with multiplicity $\geq 2$ can belong to the set $F(\On) \cap F(\Fq \setminus \On).$ However $F$ is a permutation on 
$\Fq \setminus \On$ and hence no element can have multiplicity exceeding $2$.

\item [\textit{(iv)}]
Clearly $F(x_i)= 0$ for $1 \leq i \leq n-1$ and $F(x_n)=nx_n=nb.$
The condition $a=-1$ implies $F(\Fq \setminus \On)=\{x_n\}=\{b\}.$ Therefore $|V_F|=2$ or $3$, depending on the characteristic $p$.
\end{itemize}
\end{proof}

\begin{rem}
\label{maximumcount}
The polynomials in Theorem \ref{thm:q-n}, parts \textit{(ii)} and \textit{(iii)} are evenly distributed in the following sense. 
When $F(x)$ is as in part \textit{(ii)}, 
the multiplicity
of any
non-zero element in $V_F$ 
is one.  
When $F(x)$ is as in part \textit{(iii)}, the maximum count for $V_F$ is two. 
\end{rem}

The following result is a corollary of Theorem \ref{thm:q-n}. We state it as a theorem
since it may be of independent interest. 

\begin{thm}Let $U$ be a subgroup of the multiplicative group $\Fq^*.$
Suppose that $|U|=n,~U=<\alpha>$, and $c \in \Fq^*,~c \neq 1$. Then the polynomial 
$F \in \Fqn$, defined as in (\ref{F}), with $a=-1/\alpha$ and $b=-ac$ has 
$V_F= \Fq \setminus cU$.
\end{thm}

\begin{proof} 
For $a=-1/\alpha$ and $b=-ac$,
choose $\On$ 
as in the proof of Theorem \ref{thm:q-n}\textit{(ii)}. Then the missing values in $V_F$ are exactly
those, given by 
$G(x_i)=(a+1)x_i+b=(a+1)((-b/a)\sum_{j=0}^{n-i}(-1/a)^j)+b= (-b/a)(-1/a)^{n-i}$
for $1 \leq i \leq n$. 

\end{proof}

We end this note by an example, illustrating the procedure we explained in the beginning of this section. 

\begin{ex} 
\label{Ex2}
Let $p=13$, $n=4$ and
$g(x)=-x+2$. Then $x_4=2.$ Suppose that 
$x_{i}=(1-i)x_{4}$ for $i=2,3$, as in Theorem \ref{thm:q-n}, part \textit{(iv)}. Hence 
$V_F=\{0,2,8\}.$ To describe $F(x)$, we put $\alpha_4=-\epsilon,~\beta_4=2\epsilon,~ \beta_5=\epsilon.$
Recall that $\alpha_{5}=0$. Then $c_4=11$ and 
$\alpha_{3}=\alpha_{5}-c_{4}\alpha_{4}=-2\epsilon,  \quad
\beta_{3}=\beta_{5}-c_{4}\beta_{4}=5\epsilon.$
We obtain recursively
$c_{3}=(\beta_{4}+x_{2}\alpha_{4})/(\beta_{3}+x_{2}\alpha_{3})=12,
~ \alpha_{2}=\alpha_{4}-c_{3}\alpha_{3}=10\epsilon, ~
\beta_{2}=\beta_{4}-c_{3}\beta_{3}=7\epsilon=1. $
Therefore $\epsilon=2$,
$c_{2}=\beta_{3}=10$,
$\alpha_{1}=\alpha_{3}-c_{2}\alpha_{2}=4, ~
c_{1}=\alpha_{2}/\alpha_{1}=5 ~ \textrm{and} ~
c_{0}=\alpha_{1}=4,$
giving $$F(x)=\Big(\Big(\big(((4x)^{11}+5)^{11}+10\big)^{11}+12\Big)^{11}+11\Big)^{11}+x.$$
\end{ex}

\section{Acknowledgement}
L.I. was supported by a grant in scope of the TUBITAK project number 114F432. A.T. was partially supported by TUBITAK project number 114F432.

\end {document}